\documentclass[12pt]{article}

\usepackage{epsfig,epsf}
\usepackage{amssymb,amsmath,amsthm}
\usepackage{graphicx}
\usepackage{subfigure}
\usepackage[dvipsnames]{xcolor}
\usepackage{color}
\usepackage[english]{babel}
\usepackage{epstopdf}

\usepackage{bbding}

 \makeatletter \@addtoreset{equation}{section}

\newtheorem{theorem}{Theorem}[section]
\newtheorem{lemma}{Lemma}[section]
\newtheorem{corollary}{Corollary}[section]
\newtheorem{prop}{Proposition}[section]

\newcommand\blfootnote[1]{%
	\begingroup
	\renewcommand\thefootnote{}\footnote{#1}%
	\addtocounter{footnote}{-1}%
	\endgroup
}

\begin{document}

\title{On branches of positive solutions for p-Laplacian problems  at the extreme value  of Nehari manifold method}

\author{Yavdat Il'yasov
	\and
	Kaye Silva
}

\newcommand{\Addresses}{{
		\bigskip
		\footnotesize
		
		Yavdat Il'yasov, \textsc{Institute of Mathematics, Ufa Scientific Senter, Russian Academy of Sciences, 112, Chernyshevsky str., Ufa, Russia}\par\nopagebreak
		\textit{E-mail address}: \texttt{ilyasov02@gmail.com}
		
		\medskip
		
		Kaye Silva, \textsc{Institute of Mathematics and Statistics, Federal University of Goias, Campus II, CEP 74690-900
			}\par\nopagebreak
		\textit{E-mail address}: \texttt{kayeoliveira@hotmail.com}
		
}}

\date{$\ $}
\maketitle

\abstract{This paper	is concerned with variational continuation of branches of solutions for nonlinear boundary value problems, which  involve the p-Laplacian, the indefinite  nonlinearity, and depend on  the real parameter $\lambda$. A special focus is made on the extreme value  of Nehari manifold  $\lambda^*$, which determines the threshold of applicability of Nehari manifold method.  In the main result  the existence of two branches of positive solutions for the cases where parameter $\lambda$ lies above the  threshold  $\lambda^*$ is obtained. }

\medskip

{\it Key words}: $p$-Laplacian,  extreme value of Nehari manifold method, indefinite nonlinearity, branches of solutions

\blfootnote{\Addresses}

\section{Introduction}
We study the following  p-Laplacian problem with indefinite  nonlinearity 

\begin{equation}\label{p}
\left\{
\begin{aligned}
-\Delta_p u&= \lambda |u|^{p-2}u+f|u|^{\gamma-2}u &&\mbox{in}\ \ \Omega \\
u&=0                                   &&\mbox{on}\ \ \partial\Omega 
\end{aligned}
\right.
\end{equation}

Here $\Omega$ denotes a bounded domain in $\mathbb{R}^N$ with $C^1$-boundary 
$\partial\Omega$,  $\lambda$ is a real parameter,  $1<p<\gamma < p^*$, where
$p^*$ is the critical Sobolev exponent, $f \in L^d(\Omega)$ where $d\geq 
p^*/(p^*-\gamma)$ if $p<N$,  and $d>1$ if $p\geq N$. We suppose that (\ref{p}) 
has an indefinite nonlinearity, i.e., $f$ change sign in $\Omega $.   
By a solution of (\ref{p}) we  mean  a critical point $u \in 
W:=W_0^{1,p}(\Omega)$ of the energy functional
\begin{equation*}
\Phi_\lambda (u) = \frac{1}{p} \int |\nabla u|^{p} dx - \frac{\lambda}{p} \int 
|u|^{p} dx -
\frac{1}{\gamma} \int f |u|^{\gamma} dx,
\end{equation*}
where $W_0^{1,p}(\Omega)$ is the standard Sobolev space.

The problems with the indefinite nonliearity of type \eqref{p}  have been intesively studied, see e.g.,  Alama \& Tarantello \cite{alam}, Berestycki, Capuzzo--Dolcetta \& Nirenberg \cite{BCDN}, 
Ouyang \cite{ou2}. One of the fruitful approaches in the study of such 
problems  is the Nehari manifold method \cite{Neh1} where solutions are 
obtained through the constrained minimization problem 
\begin{equation}\label{NMMg}
\min\{\Phi_\lambda(u): ~u\in \mathcal{N}_\lambda\}
\end{equation}
with the Nehari manifold $\mathcal{N}_\lambda:=\{u \in W\setminus 
0:~D_u\Phi_\lambda(u)(u)=0\}$   (see e.g.  Drabek \& Pohozhaev \cite{drabekp}, 
Il'yasov \cite{ilyas, ilyasENMM}, Ouyang \cite{ou2}).

The applicability of NM-method to \eqref{p} depends on the parameter $\lambda$. 
Indeed, \eqref{p} possess the so-called \textit{extreme value of the Nehari 
	manifold method}   \cite{ilyasENMM} 
\begin{equation}\label{Ouyang}
\lambda^*=\inf\left\{\frac{\int |\nabla u|^p dx }{\int
	|u|^pdx}:~~\int
f|u|^{\gamma}dx\geq  0,~u \in W\setminus 0\right\}, 		
\end{equation}
which was known to be the first found by Ouyang \cite{ou2}.  A feature of  
$\lambda^*$  is that it defines a threshold for the applicability of the Nehari 
manifold method  
so that  for any  $\lambda<\lambda^*$ the set $\mathcal{N}_\lambda$ is a 
$C^1$-manifold of codimension $1$ in $W$ wherein 
for any $\lambda\ge\lambda^*$ there is $u \in \mathcal{N}_\lambda$ such that 
$\Phi_\lambda''(u):=D^2_{uu}\Phi_\lambda(u)(u,u)=0$. Moreover,  
$\Phi_{\lambda}$ is unbounded from below over $\mathcal{N}_\lambda$ if $\lambda\ge\lambda^*$ (see e.g 
\cite{ilyasENMM}).
It is remarkable that once the extreme value \eqref{Ouyang} is detected, one is 
able to directly find solutions for (\ref{p}) as 
$\lambda<\lambda^*$, by means of the Nehari  minimization problems \eqref{NMMg} 
(see e.g.  \cite{ou2} for $p=2$   and  \cite{ilyas} for $1<p<+\infty$).

A natural question which arises  from this is whether there are any positive 
solutions of \eqref{p} for $\lambda>\lambda^*$. An answer for this question, in 
the case $p=2$, follows from the works of  Alama \& Tarantello  \cite{alam}, 
Ouyang 
\cite{ou2}, where the authors proved that \eqref{p} possess a branch of  
minimal positive solution for $\lambda$ belonging to the whole interval 
$(-\infty,\Lambda)$ and does not 
admit any positive solutions for $\lambda>\Lambda$.  However, one approach used in \cite{alam, 
	ou2} is based on the application of the local continuation  method \cite{CrRab}, 
which  essentially involves an analysis  of the corresponding linearized 
problems.

The main aim of the present paper is to give a contribution in  the 
investigation of the branches of solutions  for the problems where the 
application of local continuation methods can cause difficulty.
Our approach is based on the development of the Nehari manifold method  where 
we focus also  on obtaining a new knowledge on extreme value of the Nehari 
manifold method.

Let us state our main results. Denote
$$
\Omega^+=\{x\in\Omega:\ f^+(x)\neq 0\},~~ \Omega^-=\{x\in\Omega:\ f^-(x)\neq 0\}
$$
and $\Omega^0=\Omega\setminus(\overline{\Omega^+}\cup \overline{\Omega^-})$.
We write $U\neq \emptyset$ if the interior int($U$) of a set $U \subset 
\mathbb{R}^n$  is  non-empty. We denote 
($\lambda_1({\rm 
	int}(U))$, $\phi_1({\rm 
	int}(U)))$  the  first eigenpair of 
$-\Delta_p$ on ${\rm 
	int}(U)$ with zero Dirichlet boundary conditions. It is known 
that $\lambda_1({\rm 
	int}(U))$ is positive, simple and isolated, and $\phi_1({\rm 
	int}(U))$ is positive 
\cite{ lindqvist}.
To simplify notations, we write $\lambda_1:=\lambda_1(\Omega)$, $\phi_1:=\phi_1(\Omega)$.

Throughout the paper, we assume that $\Omega^+\neq \emptyset$. Furthermore, we 
shall need the following assumption
\begin{quote}
	{($f_1$):}\, If $\Omega^0\neq \emptyset$, then $\lambda_1({\rm 
		int}(\Omega^0\cup\Omega^+))<\lambda_1({\rm int}(\Omega^0))$.
\end{quote}
Notice that if $\Omega^0\cup\Omega^+ \neq \emptyset$,  then $\lambda^*<+\infty$ 
and there exists 
$\overline{\lambda}>\lambda_1$  such that (\ref{p}) has no positive solutions 
for any 
$\lambda>\overline{\lambda}$ (see e.g. \cite{ilyasCRAS}).

Our main result is the following

\begin{theorem}\label{thmlu}
	Let $1<p<\gamma < p^*$ and suppose that $\Omega^+\neq \emptyset$, 
	$F(\phi_1)<0$ and ($f_1$) is satisfied. Then there exists 
	$\Lambda>\lambda^*$ such that for all $\lambda\in (\lambda^*,\Lambda)$ 
	problem (\ref{p}) admits two positive weak solutions 
	$u_\lambda,\overline{u}_\lambda$. 
	Moreover,  
	
	\begin{description}
		\item[(i)] 
		$\Phi_\lambda''(u_\lambda)>0,\Phi_\lambda''(\overline{u}_\lambda)>0$ and 
		$\Phi_\lambda(u_\lambda)<\Phi_\lambda(\overline{u}_\lambda)<0$ 
		for any $\lambda\in 
		(\lambda^*,\Lambda)$;
		\item[(ii)] $\Phi_\lambda(u_\lambda)\uparrow 
		\Phi_{\lambda^*}(u_{\lambda^*})$ as $\lambda \downarrow \lambda^*$.
	\end{description}
\end{theorem}

This paper is organized as follows. Section 2 contains preliminaries
results. In Section 3, we show the existence of solutions $u_\lambda$. In Section 4 we show the existence of solutions $\overline{u}_\lambda$ and conclude the proof of Theorem \ref{thmlu}. In Appendix we provide some technical and auxiliary results.

\section{Preliminaries}
Denote
\begin{eqnarray*}
	H_{\lambda}(u)=
	\int|\nabla u|^p\,dx-\lambda \int|u|^p\,dx,~F(u) = \int
	f(x)|u|^{\gamma}, ~u \in W.
\end{eqnarray*}
Then 
$$
\Phi_\lambda (u) 
=\frac{1}{p}H_{\lambda}(u)-\frac{1}{\gamma}F(u),~~~\mathcal{N}_\lambda=\{u \in 
W\setminus 
0:\,H_{\lambda}(u)-F(u)=0\}.
$$
To our aims, it is sufficient to use the following Nehari submanifold 
$$
\mathcal{N}_\lambda^+:=\{u \in 
\mathcal{N}_\lambda:\,D_{uu}\Phi_\lambda(u)(u,u)> 0\}
$$
which we shall use in the fibering representation  \cite{ilyas,poh}:
\begin{eqnarray*}
	\mathcal{N}_\lambda^+=\{u=s v:  ~  s= s^+_\lambda(v),~v \in 
	\Theta^{+}_{\lambda}\},
\end{eqnarray*}
where $
\Theta^{+}_{\lambda}=\{v
\in W\setminus 0:~H_{\lambda}(v)<0,~ F(v)<0 \}
$
and
\begin{equation}
s^+_\lambda({v}) =  \left(\frac{{H}_{\lambda
	}({v})} {F(v)}\right)^{1/(\gamma
	-p)}.
\label{tt}
\end{equation}
Thus we are able to introduce
\begin{eqnarray}
J^+_{\lambda}({v})=:\Phi_\lambda (s^+_\lambda(v)
{v})=-c_{p,\gamma}\frac{|{H}_{\lambda }({v})|^{\gamma/(\gamma
		-p)}} {|{F}({v})|^{p/(\gamma
		-p)}},~ v \in \Theta^{+}_{\lambda}, \label{j2}
\end{eqnarray}
where $c_{p,\gamma}=(\gamma-p)/p\gamma$. 

Observe,  $J^+_{\lambda}$ is the 0-homogeneous functional on $\mathcal{N}_\lambda^+$, i.e.,
$J^+_{\lambda}(su)=J^+_{\lambda}(u)$ for any $s>0$, $u \in 
\mathcal{N}_\lambda^+$. 
It is worth pointing out that  \eqref{Ouyang} implies  
$$
\Theta^{+}_{\lambda}=\{v
\in  W\setminus 0:~H_{\lambda}(v)<0\}
$$ 
for any $\lambda\in (\lambda_1,\lambda^*)$. In 
what follows, we denote $\partial\Theta_\lambda^+=\{v\in W\setminus 0:\ 
H_\lambda(v)=0\}$.

It is not hard to prove  (see e.g. \cite{ilyas})
\begin{prop}\label{Prop:weak}
	If $D_v	 J^+_{\lambda}(v)(\eta)=0$ for any $\eta \in W\setminus 0$, then 
	$s^+_\lambda(v)
	v$ weakly satisfies (\ref{p}).
\end{prop}
In what follows, we shall use
\begin{prop}\label{auxiliary}
	Assume $(w_n) \subset \Theta^+_\lambda$ for $\lambda>\lambda_1$ and $||w_n||=1$, $n=1,2,...$. Then there exist $w\in W\setminus 0$ and a subsequence, which we will still denote by $(w_n)$,  such that $w_n\rightharpoonup w$ weakly in $W$ and $w_n\to w$ strongly in 
	$L^q(\Omega)$ for $1< q<p^*$.
\end{prop}
\begin{proof} Once $(w_n)$ is bounded in $W$, the proof of the existence of the limit point $w \in W$ follows from the Eberlein-\v Smulian and Sobolev theorems. Since 
	$H_{\lambda}(w_n)<0$, $n=1,2,...$, it follows that $w\neq 0$.
\end{proof}

Consider the following Nehari minimization problem
\begin{eqnarray}
\hat{J}^+_{\lambda}:=  \min	\{J^+_{\lambda}(v)|~v \in
\Theta^+_{\lambda}\}. \label{N+}
\end{eqnarray}

\begin{lemma}\label{solextr} Suppose the assumptions of Theorem \ref{thmlu} are 
	satisfied. Then

	\begin{description}
		\item[(a)]  $\lambda_1<\lambda^*<\infty$;
		\item[(b)] there exists a minimizer  $ \phi_1^*$ of the problem 
		(\ref{Ouyang}) 
		such that  $\phi_1^*>0$. Moreover, any minimizer $\phi_1^*$ of 
		(\ref{Ouyang}) weakly satisfies, up to scalar multiplier,  to \eqref{p} for 
		$\lambda=\lambda^*$ and $H_{\lambda^*}(\phi^*_1)=F(\phi^*_1)=0$; 
		\item[(c)] $\hat{J}_{\lambda^*}^+>-\infty$ and there exists a minimizer 
		$v_{\lambda^*} \in 
		\Theta^+_{\lambda^*}$ of $J_{\lambda^*}^+(v_{\lambda^*})$ so that  
		$u_{\lambda^*}:=s_2(v_{\lambda^*} )v_{\lambda^*}$ satisfies
		(\ref{p}) and $u_{\lambda^*}>0$. 
	\end{description}
	
\end{lemma}

\begin{proof} The proof of \textbf{(a)} can be found in \cite{ilyas, 
		ilyasCRAS}. 
	Furthermore, by
	\cite{ilyas},  there exists a nonzero minimizer  $\phi_1^*$ of  (\ref{Ouyang}) 
	such that $\phi_1^*\geq 0$.  Hence by Lagrange multiplier rule there exist 
	$\mu_0, \mu_1 \geq 0$, $|\mu_0|+|\mu_1|\neq 0$ such that
	\begin{equation}
	\mu_0D_v H_{\lambda^*}(\phi^*_1)=\mu_1 D_v F(\phi^*_1). 
	\end{equation}
	Since $\phi_1^*$ is a minimizer of (\ref{Ouyang}) then 
	$H_{\lambda^*}(\phi^*_1)=0$ and threfore $\mu_1 F(\phi^*_1)=0$.
	
	Suppose $\mu_0=0$, then $f|\phi^*_1|^{\gamma-2}\phi^*_1=0$ a.e. in $\Omega$. 
	This is possible only if supp $\phi^*_1$ $ \subset \Omega^0$. Thus if 
	$\Omega^0=\emptyset$ then we get a contradiciton. Assume that $\Omega^0\neq 
	\emptyset$. Then   there exist eigenpairs $(\lambda_1({\rm 
		int}(\Omega^0)),\phi_1({\rm int}(\Omega^0)))$ and $(\lambda_1({\rm 
		int}(\Omega^0\cup\Omega^+)),\phi_1({\rm int}(\Omega^0\cup\Omega^+)))$.  Since 
	$\phi_1(\Omega^0) \in W^{1,p}_0(\Omega^0)$ and $H_{\lambda^*}(\phi^*_1)=0$,  
	$\lambda_1({\rm int}(\Omega^0))\leq \lambda^*$. On the other hand, the 
	assumption ($f_1$) entails the strong inequality $\lambda_1({\rm 
		int}(\Omega^0\cup\Omega^+))<\lambda_1({\rm int}(\Omega^0))$. Hence we get a 
	contradiction because  $\lambda^*\leq \bar{\lambda} \leq \lambda_1({\rm 
		int}(\Omega^0\cup\Omega^+))$ (see \cite{ilyasCRAS}). Thus $\mu_0\neq 0$.
	
	Suppose $\mu_1=0$, then $D_u H_{\lambda^*}(\phi^*_1)=0$. By the Harnack  
	inequality (see \cite{trud}) we have $\phi^*_1 >0$ in $\Omega$. But this is  
	possible only if $\lambda^*=\lambda_1$, $\phi^*_1=\phi_1$. However, by 
	(\ref{Ouyang}), $F(\phi^*_1)\geq 0$  which contradicts the assumption 
	$F(\phi_1)<0$.
	Hence $\mu_1>0$ and therefore $F(\phi^*_1)=0$ and there exists $t(\mu)>0$ such 
	that
	$t(\mu)\phi_1^*$ satisfies  \eqref{p}. The maximum principle and regularity of 
	solutions for the $p$-Laplacian equation yields that $\phi^*_1 >0$ and 
	$\phi^*_1 \in  C^{1,\alpha}(\overline{\Omega})$. Thus we have proved 
	\textbf{(b)}.

	Let us prove \textbf{(c)}. By \cite{ilyas} there is a finite limit
	\begin{equation}\label{limGR}
	\hat{J}_{\lambda}^+ \to 	\bar{J}^+(\lambda^*)\ge -\infty~~\mbox{as}~~ 
	\lambda 
	\uparrow \lambda^*.
	\end{equation}
	and there exists a weak positive solution  $u_{\lambda^*}$ of \eqref{p} such 
	that $\bar{J}^+(\lambda^*)=J_{\lambda^*}^+(u_{\lambda^*})$. It is clear that  
	$\hat{J}_{\lambda^*}^+\leq\bar{J}^+(\lambda^*)$. Thus, we will obtain the proof 
	if we show that    $\hat{J}_{\lambda^*}^+=\bar{J}^+(\lambda^*)$. 
	Suppose, contrary to our claim, that 
	$\hat{J}_{\lambda^*}^+<\bar{J}^+(\lambda^*)$.
	We prove that this is impossible if $\hat{J}_{\lambda^*}^+=-\infty$. The proof 
	in the other case is similar. 
	
	Since $\hat{J}_{\lambda^*}^+=-\infty$, for every $K>0$, one can find  $v_K \in 
	\Theta^+_{\lambda^*}$ such that $J_{\lambda^*}^+(v_K)<\bar{J}^+(\lambda^*)-K$. 
	Since $J_{\lambda}^+(v_K) \to J_{\lambda^*}^+(v_K)$, for every $\varepsilon>0$, 
	there exists $\delta>0$ such that $|J_{\lambda}^+(v_K) - 
	J_{\lambda^*}^+(v_K)|<\varepsilon$ as $|\lambda-\lambda^*|<\delta$. In view
	of \eqref{limGR}, we may assume that there holds also $|\hat{J}_{\lambda}^+ -	
	\bar{J}^+(\lambda^*)|<\varepsilon$ if $|\lambda-\lambda^*|<\delta$. Then
	$$
	\bar{J}^+(\lambda^*)-\varepsilon<\hat{J}_{\lambda}^+\leq 
	J_{\lambda}^+(v_K)<J_{\lambda^*}^+(v_K)+\varepsilon<\bar{J}^+(\lambda^*)-K+\varepsilon.
	$$
	Since $K>0$, $\varepsilon>0$ may be chosen arbitrarily, we get a 
	contradiction.  
	
\end{proof}	

We need also
\begin{corollary}\label{cormu}
	There exists $\mu_0 \in (\lambda_1, \lambda^*)$ such that any minimizer $w_{\lambda^*}$
	of \eqref{N+} for $\lambda=\lambda^*$	satisfies $H_{\mu_0}(w_{\lambda^*})<0$.
\end{corollary}
\begin{proof}  Suppose the assertion of the corollary is false. Then there 
	exists a sequence $w_n \in 
	\Theta^+_{\lambda^*}$ such that  $\hat{J}_{\lambda^*}^+=J_{\lambda^*}^+(w_n)$ 
	and $H_{\lambda^*}(w_n) \to 0$ as $n\to \infty$. By homogeneity of $J^+_{\lambda^*}(v)$ we may assume
	$||w_n||=1$. Hence by Proposition \ref{auxiliary}, there is $w\in W\setminus 0$ such that  
	$w_n\rightharpoonup w$ in $W$ and $w_n\to w$ in 
	$L^q(\Omega)$  for $1<q<p^*$. 
	Observe
	$$
	J_{\lambda^*}^+(w_n)=-c_{p,\gamma}\frac{|H_{\lambda^*}(w_n)|^{\gamma/(\gamma-p)}}{|F(w_n)|^{p/(\gamma-p)}}=-c_{p,\gamma}s^+_{\lambda^*}(w_n)^p|H_{\lambda^*}(w_n)|.
	$$
	From this and since $J_{\lambda^*}^+(w_n)= \hat{J}_{\lambda^*}^+<0$, it follows 
	that $s^+_{\lambda^*}(w_n)\to \infty$ and  $F(w_n)\to 0$. Hence $F(w)=0$ and 
	therefore $H_{\lambda^*}(w)=0$ which implies by $(2^o)$, Lemma \ref{solextr}  
	that $w=\phi_1^*>0$.  Note that
	\begin{equation*}\label{b33}
	-\Delta_p w_n-\lambda^* 
	|w_n|^{p-2}w_n-s^+_{\lambda^*}(w_n)^{\gamma-p}f|w_n|^{\gamma-2}w_n=0,~ 
	n=1,2,\dots
	\end{equation*}
	Thus $s^+_{\lambda^*}(w_n)\to \infty$ implies $f|\phi_1^*|^{\gamma-1}=0$ a.a. 
	in $\Omega$, which is an absurd.
	
\end{proof}

\begin{corollary}\label{L01}
	For each $\mu\in (\lambda_1,\lambda^*)$, there is $c_\mu<0$ such that 
	$F(v)\le c_\mu$ $\forall v\in \overline{\Theta}_\mu^+$.
\end{corollary}

\begin{proof} Let $\mu<\lambda^*$ and assume contrary to our claim that there 
	exists a sequence $v_n \in \overline{\Theta}_\mu^+$ such that 
	$F(v_n) \to 0$ as $n \to \infty$. By homogeneity of $J^+_{\lambda^*}(v)$ and Proposition \ref{auxiliary} there is $v \in W\setminus 0$ such that
	$v_n \rightharpoonup v$ weakly in $W$ and $v_n\to v$ 
	strongly in  $L^q(\Omega)$ in 
	$L^q(\Omega)$ for $1< q<p^*$.  Hence, by the weakly 
	lower-semicontinuity of $\int |\nabla v|^p dx$ we conclude that 
	$$
	H_{\lambda^*}(v)\leq \liminf_{n\to \infty}H_{\lambda^*}(v_n)=\liminf_{n\to 
		\infty}(H_{\mu}(v_n)+(\mu-\lambda^*)\int|v_n|^p)< 0.
	$$ 
	But this contradicts to the definition of $\lambda^*$ since $F(v)=0$.  
	
\end{proof}

\section{Local Minima Solutions}

In this section, we show the existence of local minima type solutions $u_\lambda$ for \eqref{p}. Let us consider the following family of constrained minimization problems
\begin{equation}\label{consmin}
\hat{J}_\lambda^+(\mu)=\inf \{J_\lambda^+(v):\ v\in \Theta_{\mu}^+\}
\end{equation}
parametrized by $\lambda\geq \lambda^*$ and $\mu\in (\lambda_1,\lambda^*)$.

\begin{prop}\label{T1}  For each 
	$\lambda\ge\lambda^*$ and $\mu\in (\lambda_1,\lambda^*)$ there holds
	
	\begin{description}
		\item[(a)] $\hat{J}_\lambda^+(\mu)>-\infty$; 
		\item[(b)] there exists a minimizer $v_\lambda(\mu)$ 
		of \eqref{consmin}.
	\end{description}
\end{prop}

\begin{proof} \textbf{(a)} follows immediately from Lemma
	\ref{L01}. Let us prove \textbf{(b)}. Take  a minimizing sequence $v_n\in 
	\Theta_\mu^+$ of 
	\eqref{consmin}, that is  $J^+_{\lambda}(v_n)\to 
	\hat{J}_\lambda^+(\mu)>-\infty$ as $n\to \infty$. By homogeneity of $J^+_{\lambda^*}(v)$ and Proposition \ref{auxiliary} there is $v \in W\setminus 0$ such that
	$v_n \rightharpoonup v$  in 
	$W$ and $v_n\to v$  in 
	$L^q(\Omega)$ for $1< q<p^*$.   Hence, by the 
	weak lower-semicontinuity we infer that
	$H_\mu(v)\le \liminf_{n\to \infty} H_\mu(v_n)\leq 0,\ F(v)=\lim_{n\to \infty}  
	F(v_n)<0$ and therefore $v \in  \overline{\Theta_{\mu}^+}$. By the 
	weak lower-semicontinuity of $J_\lambda^+(v)$, 
	$$
	J_\lambda^+(v)\leq \liminf_{n\to 
		\infty}J_\lambda^+(v_n)=\hat{J}_\lambda^+(\mu).
	$$
	In view of \eqref{consmin}, this is possible only if 
	$J_\lambda^+(v)=\hat{J}_\lambda^+(\mu)$, that is $v$ is a minimizer 
	of  
	\eqref{consmin}. 
	
\end{proof}

We denote the set of minimizers for \eqref{consmin}  by
$
\mathcal{S}_{\lambda}(\mu)=\{v\in \overline{\Theta}_{\mu}^+:\ J_\lambda^+(v)=	
\hat{J}_\lambda^+(\mu)\}
$
and let
$
\mathcal{S}_{\lambda}^\partial(\mu)=\{v\in 
\mathcal{S}_{\lambda}(\mu):~H_\mu(v)=0\}.
$

%

\begin{lemma}\label{locminEG} 
	Let $\lambda_0 \ge \lambda^*$ and $\mu\in (\lambda_1,\lambda^*)$ such
	that 
	$\mathcal{S}^\partial_{\lambda_0}(\mu) = \emptyset$. Then there 
	exists $\varepsilon>0$ such that 
	$\mathcal{S}^\partial_{\lambda}(\mu) = \emptyset$ for 
	each $ \lambda \in 
	[\lambda_0,\lambda_0+\varepsilon)$.
\end{lemma}

\begin{proof}  Suppose the lemma were false. Then we could find  sequences  
	$\lambda_n\to \lambda_0$ such  that  $v_n:=v_{\lambda_n}^+(\mu)\in \partial\Theta_{\mu}^+$, 
	$n=1,2,...$. By By homogeneity of $J^+_{\lambda^*}(v)$ we may assume that 
	$||v_n||=1$, $n=1,2,...$, and therefore by Proposition \ref{auxiliary}, $v_n \rightharpoonup v$ 
	weakly in $W$ and $v_n\to v$ strongly  in 
	$L^q(\Omega)$ for $1< q<p^*$ and some $v\in W\setminus 0$.  Hence 
	$H_{\mu}(v)\le \liminf_{n\to \infty} H_{\mu}(v_n)= 0,\ F(v)=\lim_{n\to \infty} 
	F(v_n)<0$ and therefore $v \in  \overline{\Theta_{\mu}^+}$. Furthermore, 
	by 
	the weak lower semi-continuity 
	\begin{equation}\label{jc}
	J_{\lambda_0}^+(v)\leq \liminf_{n\to 
		\infty}J_{\lambda_n}^+(v_n)=:\tilde{J}<+\infty.
	\end{equation}
	Observe, from the Poincare's inequality and Corollary \ref{L01}, we have that 
	for all $w\in \overline{\Theta_\mu^2}$ and $\lambda \geq \lambda_1$
	$$
	|(-J^+_{\lambda}(w))^\frac{\gamma
		-p}{\gamma}-(-J^+_{\lambda_0}(w))^\frac{\gamma
		-p}{\gamma}|=\frac{|\lambda-\lambda_0|G(w)}{|F(w)|^{p/\gamma
	}}\le \frac{|\lambda-\lambda_0|}{\lambda_1}\frac{1}{|c_\mu|^{p/\gamma}}.
	$$
	Thus, $J_{\lambda_n}^+(w) \to J_{\lambda_0}^+(w)$ uniformly on 
	$w\in \overline{\Theta_{\mu}^+}$ as $n \to \infty$ and therefore 
	$\tilde{J}=\hat{J}_{\lambda_0}^+(\mu)$. Hence if 
	$J_{\lambda_0}^+(v)<\hat{J}_{\lambda_0}^+(\mu)$, we obtain a 
	contradiction 
	since $v \in  \overline{\Theta_{\mu}^+}$.  The case 
	$J_{\lambda_0}^+(v)=\hat{J}_{\lambda_0}^+(\mu)$ entails that 
	$H_{\mu_0}(v)= 0$ 
	and $v=v_{\lambda_0}(\mu)$. Consequently $v_{\lambda_0}(\mu) \in 
	\mathcal{S}_{\lambda_0}^\partial(\mu)$ which contradicts to the assumption  
	$\mathcal{S}_{\lambda_0}^\partial(\mu)= \emptyset$.

\end{proof}

Let us prove the existence of the first solution $u_\lambda$ in Theorem \ref{thmlu}.
\begin{lemma}\label{lemlu}
	Let $1<p<\gamma < p^*$ and suppose that $\Omega^+\neq \emptyset$, 
	$F(\phi_1)<0$ and ($f_1$) is satisfied. Then there exists 
	$\Lambda>\lambda^*$ such that for all $\lambda\in (\lambda^*,\Lambda)$ 
	problem (\ref{p}) admits  positive weak solution 
	$u_\lambda$ such that 
	\begin{description}
		\item[(li)] 
		$\Phi_\lambda''(u_\lambda)>0$ and 
		$\Phi_{\lambda}(u_{\lambda})<0$ 
		for any $\lambda\in 
		(\lambda^*,\Lambda)$;
		\item[(lii)] $\Phi_\lambda(u_\lambda)\uparrow 
		\Phi_{\lambda^*}(u_{\lambda^*})$ as $\lambda \downarrow \lambda^*$;
		
	\end{description}
\end{lemma}
\begin{proof}
	From Lemma \ref{cormu} it follows that
	there exists $\mu_0 \in (\lambda_1, \lambda^*)$ such that $\mathcal{S}_{\lambda^*}^\partial(\mu_0)= \emptyset$. Thus Lemma \ref{locminEG} implies that there exist $\Lambda>\lambda^*$  such that $\mathcal{S}_{\lambda}^\partial(\mu_0)= \emptyset$
	for all $\lambda\in (\lambda^*,\Lambda)$. Since by Proposition \ref{T1}, $\mathcal{S}_{\lambda}(\mu)\neq \emptyset$ for $\lambda\geq \lambda^*$, we conclude that for every $\lambda\in (\lambda^*,\Lambda)$
	there exists a minimizer $v_\lambda(\mu_0)$ of \eqref{consmin} such that $v_\lambda(\mu_0) \in \Theta^+_{\mu_0}$. This and Proposition \ref{Prop:weak} yield that $u_\lambda=s^+_\lambda(v_\lambda(\mu_0))v_\lambda(\mu_0)$ is a weak solution of \eqref{p} for $\lambda\in (\lambda^*,\Lambda)$.

	In virtue that $J_\lambda^+(v)=J_\lambda^+(|v|)$ and $|v| \in  \Theta_{\mu_0}^+$ for any $v \in  \Theta_{\mu_0}^+$, we may assume that $u_\lambda\geq 0$ in $\Omega$. Now by  the Harnack  inequality (see \cite{trud}) we conclude that $u_\lambda^+>0$ in $\Omega$. 
	
	Since $v_\lambda^+\in \Theta_\lambda^+$, we get \textbf{(li)}. Let us prove assertion \textbf{(lii)}. Notice that from \eqref{consmin} it follows that $\Phi_{\lambda^*}(u_{\lambda^*})=\hat{J}_{\lambda^*}^+(\mu_0) \geq \hat{J}_\lambda^+(\mu_0)$ for any $\lambda>\lambda^*$. Thus if we suppose that  assertion \textbf{(lii)} were false then we could find a sequence $\lambda_n \downarrow \lambda^*$ such that $J_{\lambda_n}^+(v_{\lambda_n}(\mu_0)) \to J_{\lambda^*}^+<\hat{J}_{\lambda^*}^+(\mu_0)$. Arguing as above, we may assume that $v_{\lambda_n}(\mu_0) \rightharpoonup v$ weakly in $W$ and $v_n\to v$ strongly in $L^p(\Omega)$, $L^\gamma(\Omega)$ as $n\to \infty$ with $v \neq 0$. This implies that $v \in  \overline{\Theta_{\mu_0}^+}$ and $J_{\lambda^*}^+(v)<\hat{J}_{\lambda^*}^+(\mu_0)$. Thus we get a contradiction.
\end{proof}

From the proof of Lemma \ref{lemlu} we see that the solution $u_\lambda$  may 
depend on the parameter $\mu \in (\lambda_1, \lambda^*)$. However, one can 
prove that, at least locally by $\mu$, there is no such dependence.
\begin{corollary}\label{qqlpart}
	Let $\lambda\ge \lambda^*$ and $\mu_0\in (\lambda_1,\lambda^*)$. Suppose 
	that 
	$\mathcal{S}^\partial_{\lambda}(\mu_0) = \emptyset$. Then there exists 
	$\varepsilon>0$ such that  
	$\mathcal{S}_{\lambda}(\mu_0)=\mathcal{S}_{\lambda}(\mu)$ for all $ \mu \in 
	(\mu_0-\varepsilon, \mu_0+\varepsilon)$.
\end{corollary} 
\begin{proof} Conversely, suppose that there is $(\mu_n)$ such that $\mu_n \to 
	\mu_0$ and $\exists v_n \in \mathcal{S}_{\lambda}(\mu_n) \setminus 
	\mathcal{S}_{\lambda}(\mu_0)$.  Then 
	$J_{\lambda}^+(v_n)<\hat{J}_{\lambda}^+(\mu_0)$ and $v_n \in 
	\Theta^+_{\mu_n} 
	\setminus \Theta^+_{\mu_0}$.  Arguing as in the proof of Proposition 
	\ref{T1} 
	it can be shown that there exists a subsequence (which we denote again 
	$(v_n)$) 
	such that $v_n \to v$ strongly in $W^{1,2}$. Hence, 
	$J_{\lambda}^+(v)=\hat{J}_{\lambda}^+(\mu_0)$ and $v \in \partial 
	\Theta^+_{\mu_0}$ that is $v \in \mathcal{S}^\partial_{\lambda}(\mu_0)$ 
	which 
	is a contradiction.
\end{proof}	
\section{Mountain Pass Solutions}

In this section, for $\lambda\in (\lambda^*,\Lambda)$, we will find the second
branch of positive solution $\bar{u}_\lambda$ of a mountain pass type. 

Fix $\mu_0 \in (\lambda_1, \lambda^*)$ such that any minimizer $w_{\lambda^*}$
of \eqref{N+} for $\lambda=\lambda^*$	satisfies $H_{\mu_0}(w_{\lambda^*})<0$. The existence of $\mu_0$ follows from Corollary \ref{cormu}. 

Let $\lambda\in (\lambda^*,\Lambda)$. Define 
\begin{equation}\label{mus}
\mu^{\lambda}=\sup \{\mu\in (\mu_0,\lambda^*):\ 
\hat{J}_\lambda^+(\mu)=\hat{J}_\lambda^+(\mu_0)\},
\end{equation}

\begin{prop}\label{mppro} For each $\lambda\in (\lambda^*,\Lambda)$ there 
	holds
	
	\begin{description}
		\item[(a)] $\mu_0<\mu^\lambda<\lambda^*$; 
		\item[(b)]$ \hat{J}^{+}_\lambda(\mu^\lambda)= \hat{J}^{+}_\lambda(\mu_0)$ and $\mathcal{S}_\lambda^\partial (\mu^\lambda)\neq\emptyset$.
	\end{description}
	
\end{prop}

\begin{proof} \textbf{(a)} By Corollary \ref{cormu},  $\mathcal{S}_{\lambda^*}^\partial(\mu_0)= \emptyset$ and by Corollary \ref{qqlpart}, $\mathcal{S}_{\lambda}(\mu_0)=\mathcal{S}_{\lambda}(\mu)$ for  $ \mu \in 
	(\mu_0, \mu_0+\varepsilon)$ and some $\varepsilon>0$. Hence, 
	$\mu_0<\mu^\lambda$. Notice that by  Proposition \ref{continui} from Appendix, the function $\hat{J}_\lambda^{+}(\mu)$ is continuous with respect to $\mu \in (\mu_0, \lambda^*)$. Hence and since $\hat{J}_\lambda^{+}(\mu)\to 
	-\infty$ as $\mu\to \lambda^*$, there is $\mu'\in 
	(\mu_0,\lambda^*)$ such that 
	$\hat{J}^{+}_\lambda(\mu_0)>\hat{J}^{+}_\lambda(\mu)$ for each $\mu\in 
	(\mu',\lambda^*)$. Thus 
	$\mu^\lambda\leq\mu'<+\infty$.

	\textbf{(b)} Continuity of $\hat{J}_\lambda^{+}(\cdot)$ and  \eqref{mus} yield $ \hat{J}^{+}_\lambda(\mu^\lambda)= 
	\hat{J}^{+}_\lambda(\mu_0)$. Suppose, contrary to our claim, that  $\mathcal{S}_\lambda^\partial 
	(\mu^\lambda)=\emptyset$. Then by 
	Corollary \ref{qqlpart},  there is $\varepsilon'>0$ such that for  $ \mu \in 
	(\mu^\lambda, \mu^\lambda+\varepsilon')$, 
	$\mathcal{S}_{\lambda}(\mu^\lambda)=\mathcal{S}_{\lambda}(\mu)$ and consequently $ \hat{J}^{+}_\lambda(\mu)=\hat{J}^{+}_\lambda(\mu^\lambda)= 
	\hat{J}^{+}_\lambda(\mu_0)$ which is a contradiction.

\end{proof}

Observe, for any $\lambda\in (\lambda^*,\Lambda)$ and $\mu \in (\lambda_1,\lambda^*)$, if $w \in \mathcal{S}_\lambda^\partial(\mu^\lambda)$, then $|w| \in \mathcal{S}_\lambda^\partial(\mu^\lambda)$

For each $\lambda\in (\lambda^*,\Lambda)$, fix $0\le w_\lambda\in 
\mathcal{S}_\lambda^\partial(\mu^\lambda)$ and let $0< u_\lambda\in 
\Theta_{\mu_0}^+$ be the local 
minimum found in Lemma \ref{lemlu}. Define
\begin{equation}\label{MP}
c_\lambda=\inf_{\eta\in\Gamma_\lambda}\max_{t\in [0,1]}\Phi_\lambda(\eta(t)),
\end{equation}
where 
$$
\Gamma_\lambda=\{\eta\in C([0,1],W): 
\eta(0)=u_\lambda,\ 
\eta(1)=w_\lambda\}.
$$ 
%

%
%
%
%

\begin{prop}\label{boundlimi}
	For each $\lambda\in (\lambda^*,\Lambda)$ there exists $j_\lambda$ such that 
	\begin{equation*}
	\Phi_\lambda(u)\ge 
	j_\lambda>\hat{J}_\lambda^+(\mu_0),~~~\ \forall\ 
	u \in \partial \Theta_{\mu_0}^+.
	\end{equation*}
\end{prop}

\begin{proof} Evidently, $\mathcal{S}^\partial_\lambda(\mu_0)=\emptyset$ implies 
	$$
	j_\lambda :=\inf\{J_\lambda^+(v):\ v\in \partial 
	\Theta_\mu^+\}>\hat{J}^{+}_\lambda(\mu_0),~~~ \forall\ \lambda\in (\lambda^*,\Lambda).
	$$
	Thus for any $u \in \partial {\Theta}_{\mu_0}^+$, one has
	$$
	\Phi_\lambda(u)\ge \Phi_\lambda(s_\lambda^+(u)u)=J_\lambda^+(u)\ge  
	j_\lambda>\hat{J}_\lambda^+(\mu_0).
	$$  
\end{proof}

Let us shows that every path from $\Gamma_\lambda$ intersects 
$\partial {\Theta}_{\mu_0}^+$.

\begin{prop}\label{boundcrs}
	Let $\lambda\in(\lambda^*,\Lambda)$.	Then for any $\eta\in \Gamma_\lambda$ 
	there exists $t_0\in (0,1)$ such that 
	$\eta(t_0)\in 	\partial {\Theta}_{\mu_0}^+$. 
\end{prop}

\begin{proof} Notice $H_{\mu_0} (\eta(0))=H_{\mu_0} (v_\lambda)<0$ while 
	$H_{\mu_0} (\eta(1))=H_{\mu_0}(w_\lambda)>0$ because $w_\lambda\in 
	\Theta_{\lambda^*}^+\setminus \overline{\Theta}_{\mu_0}^+$. Thus by the continuity of $H_{\mu_0}(\eta(\cdot))$, there is $t_0\in (0,1)$ 
	such that $H_{\mu_0}(\eta(t_0))=0$. 
	
	%
	
\end{proof}

Using \cite{diaz} we are able to prove 
\begin{prop}\label{ltop2}
	For each $\lambda\in (\lambda^*,\Lambda)$, there is $\overline{\eta}\in 
	\Gamma_\lambda$ such that $H_{\lambda^*}(\overline{\eta}(t))<c<0$ 
	for all $t\in [0,1]$
\end{prop}

\begin{proof} Consider the path $\overline{\eta}(t)=[(1-t)v_\lambda^p+tw_\lambda^p]^{1/p}$, $t\in [0,1]$. Once $v_\lambda>0$,  $
	\{x\in\Omega:\ u_\lambda(x)=w_\lambda(x)=0\}=\emptyset$.
	Hence we may apply 
	Proposition \ref{saadia} from the Appendix and thus  $\overline{\eta}\in C([0,1],W)$ and 
	for $ t\in [0,1]$ we have
	\begin{align*}
	H_{\lambda^*}(\overline{\eta}(t))&=\int |\nabla 
	\overline{\eta}(t)|^p-\lambda^*\int |\overline{\eta}(t)|^p \\
	&\le (1-t)\int |\nabla v_\lambda|^p+t\int 
	|\nabla w_\lambda|^p-\lambda^*\left((1-t)\int 
	|v_\lambda|^p+t \int |w_\lambda|^p\right) \\
	& 
	=(1-t)H_{\lambda^*}(v_\lambda)+tH_{\lambda^*}(w_\lambda)
	\le H_{\lambda^*}(v_\lambda)+H_{\lambda^*}(w_\lambda) \\
	&< H_{\mu_0}(v_\lambda)+(\mu_0-\lambda^*)\int |v_\lambda|^p 
	\le (\mu_0-\lambda^*)\int |v_\lambda|^p<0.
	\end{align*}

\end{proof}

\begin{corollary}\label{criticalenergysign}
	For all $\lambda\in(\lambda^*,\Lambda)$ there holds 
	\begin{equation}\label{ner}
	\hat{J}_\lambda^+(\mu_0 )< c_\lambda<0.
	\end{equation}
\end{corollary}

\begin{proof} Let us start with the first inequality. Take any $\eta\in 
	\Gamma_\lambda$. From Proposition \ref{boundcrs}, there is $t_0\in (0,1)$ 
	such that $\eta(t_0)\in  \partial {\Theta}_{\mu_0}^+$, 
	therefore by Proposition \ref{boundlimi}, $\max_{t\in 
		[0,1]} \Phi_\lambda(\eta(t))\ge \Phi_\lambda(\eta(t_0))>\hat{J}_\lambda^+(\mu_0)$ and 
	consequently $\hat{J}_\lambda^+(\mu_0)<c_\lambda$. Let  $\overline{\eta}$ be given by
	Proposition \ref{ltop2}. Then 
	$$
	\Phi_\lambda(\overline{\eta}(t))=J_\lambda^+(\overline{\eta}(t))<0, \forall\ t\in [0,1].
	$$
	which implies that $c_\lambda<0$.
	
\end{proof}

Now we are able to find the second solution $\bar{u}_\lambda$.

\begin{lemma}\label{mppp13} For each $\lambda\in(\lambda^*,\Lambda)$,  
	$c_\lambda<0$ is a critical value $\overline{u}_\lambda$ of $\Phi_\lambda$ such that 
	$\Phi_\lambda(s_\lambda^+(\overline{u}_\lambda)=c_\lambda$,  
	$\overline{u}_\lambda$ 
	is a weak 
	solution 
	of (\ref{p}) and 
	$\overline{u}_\lambda>0$ in $\Omega$.
	\par
\end{lemma}

\begin{proof} Since  $\Phi_\lambda(u)=\Phi_\lambda(|u|)$ for all $u\in 
	W$, then by (\ref{MP}) there is a sequence of paths $\eta_n\ge 
	0$ in $\Omega$ such that 
	\begin{equation*}
	\lim_{n\to \infty}\max_{t\in [0,1]}\Phi(\eta_n(t))= c_\lambda.
	\end{equation*}  
	Following \cite{kuzpoh}, for each $\epsilon>0$  introduce 
	
	\begin{equation*}
	\eta_{n,\epsilon}=\{u\in W: \inf_{t\in [0,1]} 
	\|u-\eta_n(t)\|\le \epsilon\}\cap K_{c_\lambda,2\epsilon},
	\end{equation*}
	where $K_{c_\lambda,2\epsilon}=\{u\in W: 
	|\Phi_\lambda(u)-c_\lambda|\le 2\epsilon\}$. By Theorem E.5 from \cite{kuzpoh}, 
	there is a sequence $u_n\in W$ satisfying 
	\begin{equation}\label{ps1}
	\Phi_\lambda(u_n)\to c_\lambda,\ D_u\Phi_\lambda(u_n)\to 0,
	\end{equation}
	and 
	\begin{equation}\label{ps2}
	\inf_{t\in [0,1]}\|u_n-\eta_n(t)\|\to 0.
	\end{equation}
	By Corollary \ref{criticalenergysign} we know that $c_\lambda<0$. Thus, 
	by (\ref{ps1}) we can apply Proposition \ref{la1c} to conclude that $u_n\to 
	\overline{u}_\lambda\in W\setminus 0$ so that $\Phi_\lambda(\overline{u}_\lambda)=c_\lambda$ and 
	$D_u\Phi_\lambda(\overline{u}_\lambda)=0$. Moreover, once (\ref{ps2}) is satisfied, we also have that $\overline{u}_\lambda\ge 
	0$. Now applying the Harnack inequality \cite{trud} we deduce that 
	$\overline{u}_\lambda>0$ in $\Omega$.
	
\end{proof}

\begin{proof}[Conclusion of the proof of Theorem \ref{thmlu}]
	
	Let $\Lambda>\lambda^*$ be given by Lemma \ref{lemlu}. Then Lemma \ref{lemlu} and Lemma \ref{mppp13} yield 		the existence of positive weak solutions 
	$u_\lambda,\overline{u}_\lambda$. Since $c_\lambda<0$, $\Phi_\lambda(\overline{u}_\lambda)<0$. Thus in virtue that 
	$\overline{u}_\lambda$ is a critical value of $\Phi_\lambda$, implies that $\Phi_\lambda''(\overline{u}_\lambda)>0$. Corollary \ref{criticalenergysign} and Lemma \ref{mppp13} imply that $\Phi_\lambda(u_\lambda)<\Phi_\lambda(\overline{u}_\lambda)<0$ 	for any $\lambda\in 	(\lambda^*,\Lambda)$. Hence and by \textbf{(li)}, Lemma \ref{lemlu} we get assertion \textbf{(i)} of  the theorem. The proof of \textbf{(ii)} follows from \textbf{(lii)}, Lemma \ref{lemlu}. 
	
\end{proof}

\section{Appendix}

\begin{prop}\label{continui}
	For any $\lambda\ge \lambda^*$, the function $\mu \mapsto 
	\hat{J}^+_\lambda(\mu)$ is continuous over the interval 
	$(\lambda_1,\lambda^*)$. \end{prop}	

\begin{proof}
	Let $\mu \in (\lambda_1,\lambda^*)$. Suppose, contrary to our claim, that there are $\mu_n\to \mu$ and $r>0$ such that 
	$|\hat{J}^+_\lambda(\mu_n)- \hat{J}^+_\lambda(\mu)|>r$ for all $n$, or 
	equivalently 
	
	\begin{equation}
	\label{Ap11}
	\hat{J}^+_{\lambda}(\mu_n)> \hat{J}_{\lambda}^+(\mu)+r 
	~~\mbox{or}~~\hat{J}_{\lambda}^+(\mu)> \hat{J}^+_{\lambda}(\mu_n)+r,
	\end{equation}
	for sufficiently large $n$. Suppose the first inequality is true, i.e., 
	$\hat{J}^+_{\lambda}(\mu_n)> \hat{J}_{\lambda}^2(\mu)+r$. From \eqref{consmin} this is possible only if $\mu_n<\mu$. 
	Moreover, we can assume without loss of generality that $\mu_n$ is monotone 
	increasing and consequently $\hat{J}_\lambda^2(\mu_n)$ is decreasing. Thus  
	$\hat{J}_\lambda^+(\mu_n)\to I>\hat{J}_{\lambda}^+(\mu)$.
	
	By Proposition \ref{T1}, there is $v\in  \mathcal{S}_\lambda(\mu)$ that is 
	$J_\lambda^2(v)=\hat{J}_\lambda^+(\mu)$.  Suppose $v\in 	
	\mathcal{S}_\lambda(\mu)\setminus \mathcal{S}^\partial_\lambda(\mu)$, then 
	convergence $\mu_n\to \mu$ entails that there is $n$ such that $v\in 
	\Theta_{\mu_n}^2$. However $J_\lambda^+(v)\geq \hat{J}_\lambda^+(\mu_n)$ which 
	contradicts to   $ J_\lambda^+(v)=\hat{J}_\lambda^+(\mu)<I\le 
	\hat{J}_\lambda^+(\mu_n)$. Suppose now that $v\in 
	\mathcal{S}^\partial_\lambda(\mu)$. Then, taking into account the continuity of $J_\lambda^+(u)$ on $\Theta_{\mu}^+$, we can choose $w\in \Theta_{\mu}^+$ 
	such that $J_\lambda^+(v)\le J_\lambda^+(w)<I$. However, there is $n$ such that 
	$w\in \Theta_{\mu_n}^+$. This implies $\hat{J}^+_\lambda(\mu_n)\le 
	J_\lambda^+(w)<I$ which is an absurd.
	
	Now suppose the second inequality in \eqref{Ap11} is true. Then $\mu<\mu_n$ and 
	we may assume that $\mu_n$ is decreasing. Consequently 
	$\hat{J}_\lambda^+(\mu_n)$ is increasing and  $\hat{J}_\lambda^+(\mu_n)\to I< 
	\hat{J}_\lambda^+(\mu)$. From  Proposition \ref{T1}, there is $v_n$ such 
	that $v_n\in  \mathcal{S}_\lambda(\mu_n)$. If $v_n\in \Theta_\mu^+$ for some 
	$n$ then $\hat{J}_\lambda^+(\mu)\le J_\lambda^+(v_n)=\hat{J}_\lambda^+(\mu_n)$ 
	which is contradicts to the assumption $\hat{J}_{\lambda}^+(\mu)> 
	\hat{J}^+_{\lambda}(\mu_n)$. Thus it is only possible that $v_n\in 
	\overline{\Theta}^+_{\mu_n}\setminus \overline{\Theta_\mu^+}$ for all 
	$n=1,2,...$.
	
	By homogeneity of $J^+_{\lambda}(v)$ we may assume
	$||v_n||=1$. Hence by Proposition \ref{auxiliary},   $v_n \rightharpoonup v$ in 
	$W$, $v_n \to v$ in $L^p(\Omega),L^\gamma(\Omega)$ for some $v \in W\setminus 0$. By the weak lower-semicontinuity we have that  
	$$
	H_{\mu}(v)\le\liminf_{n\to \infty} H_{\mu_n}(v_n)\le 0,\ F(v)=\lim_{n\to \infty} F(v_n)<0,
	$$ 
	which 
	implies that $v\in \overline{\Theta_{\mu}^+}$ and
	$$
	\hat{J}_\lambda^+(\mu)\le J_{\lambda}^+(v)\le \liminf_{n\to \infty} 
	J_{\lambda_n}^+(v_n)=I,
	$$
	which is an absurd because $I< \hat{J}_\lambda^+(\mu)$.
	
\end{proof}

The next result can be found in \cite{diaz}. We give a proof  here for the 
reader's convenience.

\begin{prop}\label{saadia}
	Let $u,v\in W \setminus 0$,  $u,v\ge 0$ in $\Omega$ and define 
	$\overline{\eta}(t)=[(1-t)u^p+tv^p]^{1/p}$ for $t\in [0,1]$. Suppose that 
	the set $\{x\in\Omega:\ u(x)=v(x)=0\}$ has zero Lebesuge measure. Then  
	$$|\nabla \overline{\eta}(t)|^p\le (1-t)|\nabla u|^p+t|\nabla v|^p,\ 
	\forall\ t\in [0,1],~~\mbox{a.e. in}~ \Omega
	$$
	
	\noindent and $\overline{\eta}\in C([0,1],W)$.
\end{prop}

\begin{proof} First note that the weak derivative of $\eta(t)$ is given by 
	$\nabla \eta (t)=[(1-t)u^p+tv^p]^{(1-p)/p}[(1-t)u^{p-1}\nabla u+tv^{p-1}\nabla 
	v]$. Let $p'$ be the conjugate exponent of $p$, i.e., $1/p+1/p'=1$. From the 
	Holder inequality, we have that 
	\begin{align}\label{ktp}
	&|\nabla \overline{\eta}(t)|\le [(1-t)u^p+tv^p]^{(1-p)/p}[(1-t)u^{p-1}|\nabla 
	u|+tv^{p-1}|\nabla v|]	\\
	&= [(1-t)u^p+tv^p]^{(1-p)/p}[(1-t)^{1/p'}u^{p-1}(1-t)^{1/p}|\nabla 
	u|+t^{1/p'}v^{p-1}t^{1/p}|\nabla v|] \nonumber \\
	&\le [(1-t)u^p+tv^p]^{(1-p)/p}[(1-t)u^p+tv^p]^{1/p'}[(1-t)|\nabla u|^p+t|\nabla 
	v|^p]^{1/p}~\mbox{a.e. in}~ \Omega \nonumber 
	\end{align}
	for all $t\in [0,1]$.
	Once $\{x\in\Omega:\ u(x)=v(x)=0\}$ has zero Lebesuge measure, we have that 
	$(1-t)u+tv>0$ a.e. in $\Omega$, for all $t\in (0,1)$ and therefore, from (\ref{ktp}), we 
	conclude that 
	$$
	|\nabla \overline{\eta}(t)|\le [(1-t)|\nabla u|^p+t|\nabla v|^p]^{1/p},\ 
	\forall\ t\in [0,1], ~\mbox{a.e. in}~ \Omega
	$$
	which implies $$|\nabla\overline{\eta} (t)|^p\le (1-t)|\nabla 
	u|^p+t|\nabla v|^p,\ \forall\ t\in [0,1], ~\mbox{a.e. in}~ \Omega.
	$$
	Consequently $\overline{\eta}\in W$ for all $t\in 
	[0,1]$. The continuity of $\eta$ follows by a standard application of the 
	Lebesgue theorem.
	
\end{proof}

\begin{prop}\label{la1c}
	
	Suppose that $u_n\in W\setminus 0$ is a (P.-S.) sequence, 
	i.e.
	$$ 
	\Phi_\lambda(u_n)\to c< 0,~~ D_u\Phi(u_n)\to 0.
	$$
	Then $u_n$ has a strong convergent subsequence with non-zero limit point $ 
	u\in W\setminus 0$ satisfying $\Phi_\lambda(u)=c$ and 
	$D_u\Phi(u)=0$. 
	
\end{prop}

\begin{proof}
	The assumption $D_u\Phi_\lambda(u_n)\to 0$ entails  $H_\lambda(u_n)-F(u_n)=o(1)$ and 
	therefore
	\begin{equation}\label{Pr745}
	\frac{1}{p}H_\lambda(u_n)-\frac{1}{\gamma}F(u_n)=\frac{\gamma-p}{p\gamma}H_\lambda(u_n)+o(1)\to
	c<0~~\mbox{as}~~ n\to \infty.
	\end{equation}
	This  implies that $\|u_n\|$ is bounded, $||u_n||\geq \delta >0$ and 
	$H_\lambda(u_n)<0$ for sufficiently large $n$. 
	Thus we may assume $u_n\rightharpoonup u$ in $W$, $u_n \to u$ 
	in $L^p(\Omega)$ and $L^\gamma(\Omega)$ and $u\neq 0$.
	Hence and since $D_u\Phi(u_n)\to 0$ as $n\to \infty$,  we have
	$$
	\limsup_{n\to \infty} \langle -\Delta_p u_n, u_n-u\rangle =0.
	$$
	Thus by $S^+$ property of the $p$-Laplacian operator (see \cite{drabekMilota}) we derive that $u_n\to u$ 
	strongly in $W$.
	
\end{proof}


\end{document}